\documentclass[12pt,a4]{amsart}
\usepackage{amsmath,amssymb}
\title[Quasi-homomorphism rigidity]{Quasi-homomorphism rigidity\\
with noncommutative targets}
\author[N. Ozawa]{Narutaka OZAWA}
\address{Department of Mathematical Sciences,
University of Tokyo, \mbox{153-8914}}
\email{narutaka@ms.u-tokyo.ac.jp}
\thanks{Partially supported by JSPS and the Hausdorff Research Institute for Mathematics}

\date{July 31, 2010}
\newtheorem{thm}{Theorem}
\newtheorem{lem}[thm]{Lemma}
\newtheorem{prop}[thm]{Proposition}
\newtheorem{cor}[thm]{Corollary}
\newtheorem{thmA}{Theorem}

\theoremstyle{definition}
\newtheorem*{defn}{Definition}
\newtheorem*{rem}{Remark}
\newcommand{\IR}{{\mathbb R}}
\newcommand{\IC}{{\mathbb C}}
\newcommand{\IN}{{\mathbb N}}
\newcommand{\IZ}{{\mathbb Z}}
\newcommand{\IK}{{\mathbb K}}
\newcommand{\IB}{{\mathbb B}}
\newcommand{\G}{\Gamma}
\newcommand{\e}{\varepsilon}
\newcommand{\p}{\varphi}
\newcommand{\cH}{{\mathcal H}}
\newcommand{\cK}{{\mathcal K}}
\newcommand{\cM}{{\mathcal M}}
\newcommand{\cU}{{\mathcal U}}
\DeclareMathOperator*{\Lim}{Lim}
\DeclareMathOperator*{\esup}{ess-sup}

\newcommand{\cb}{\mathrm{cb}}
\newcommand{\SL}{\mathrm{SL}}
\newcommand{\EL}{\mathrm{EL}}
\newcommand{\ip}[1]{\mathopen{\langle}#1\mathclose{\rangle}}

\DeclareMathOperator{\EXP}{EXP}

\newcommand{\TTT}{\ensuremath{(\mathrm{TTT})}}
\newcommand{\Tp}{\ensuremath{(\mathrm{T_P})}}
\newcommand{\Tq}{\ensuremath{(\mathrm{T_Q})}}
\newcommand{\wh}{\ensuremath{(\mathrm{h})}}
\newcommand{\cF}{{\mathcal F}}

\DeclareMathAlphabet{\mathpzc}{OT1}{pzc}{m}{it}
\newcommand{\cc}{{\mathpzc b}}
\newcommand{\el}{\ell}
\newcommand{\oh}{1/2}
\begin{document}
\begin{abstract}
As a strengthening of Kazhdan's property $\mathrm{(T)}$
for locally compact groups, property $\mathrm{(TT)}$ was
introduced by Burger and Monod.
In this paper, we add more rigidity and introduce property $\mathrm{(TTT)}$.
This property is suited for the study of rigidity phenomena for
quasi-homomorphisms with noncommutative targets.
Partially upgrading a result of Burger and Monod,
we will prove that $\mathrm{SL}_n({\mathbb R})$ with $n\geq3$ and
their lattices have property $\mathrm{(TTT)}$.
As a corollary, we generalize the well-known fact that
every homomorphism from such a lattice into an amenable group or
a hyperbolic group has finite image to the extent that
it includes a quasi-homomorphism.
\end{abstract}
\maketitle
\section{Introduction}\label{sec:intro}
It is proved by Burger and Monod (\cite{burger-monod,burger-monod2}) that
lattices in higher rank Lie groups have property (TT),
a property which strengthens Kazhdan's property (T) and
implies triviality of quasimorphisms.
See \cite{bhv} for a thorough treatment of property (T), and
Section 13 in \cite{monod} for property (TT).
The purpose of this paper is to introduce a yet stronger
variant of property (TT), which we call property \TTT.
Throughout this paper \textbf{all groups are assumed to be second countable}.

\begin{defn}
Let $G$ be a locally compact group, and consider
a Borel map $\cc$ from $G$ into a Hilbert space $\cH$,
together with a Borel map $\pi$ from $G$ into the unitary group $\cU(\cH)$.
We assume that $\cc$ is \emph{locally bounded}, i.e.,
it is bounded on every compact subset.
The map $\cc$ is a \emph{cocycle} if $\pi$ is a representation
and $\cc$ satisfies
$\cc(gh)=\cc(g)+\pi(g)\cc(h)$ for all $g,h\in G$.
It is a \emph{quasi-cocycle} if $\pi$ is a representation
and the \emph{defect of $\cc$} is finite:
\[
\sup_{g,h\in G}\|\cc(gh)-\big(\cc(g)+\pi(g)\cc(h)\big)\|<+\infty.
\]
It is a \emph{wq-cocycle} if the defect is finite
(and no multiplicativity condition on $\pi$).
Recall that $G$ has property $\mathrm{(T)}$ (resp.\ $\mathrm{(TT)}$)
if every cocycle (resp.\ quasi-cocycle) on $G$ is bounded.
We say $G$ has \emph{property \TTT} if every wq-cocycle on $G$ is bounded.
\end{defn}

The study of property {\TTT} is motivated by the following fact.
A map $q\colon G\to G'$ is called a \emph{quasi-homomorphism} if
it is a continuous map (or just Borel and locally bounded)
such that the defect $\{ q(gh)^{-1}q(g)q(h) : g,h\in G \}$
is relatively compact in $G'$.
In the case where $G'=\IR$, quasi-homomorphisms are often
called quasimorphisms and have been studied extensively.
(See \cite{calegari}.)
It is easily seen that the composition $\cc\circ q$ of
a quasi-homomorphism $q\colon G\to G'$ and
a wq-cocycle $\cc\colon G'\to\cH$ is again a wq-cocycle.

\begin{defn}
We say a locally compact group $G$ is \emph{a-TTT-menable}
(or $G$ has \emph{property \wh}) if there is a wq-cocycle $\cc$ on $G$
which is \emph{proper} in the sense that
$\{ g\in G : \|\cc(g)\|\le C \}$ is relatively compact for every $C>0$.
\end{defn}

Groups with proper cocycles (i.e., a-T-menable groups,
also known as groups with Haagerup's property, see \cite{ccjjv})
are a-TTT-menable.
In particular, all amenable groups are a-TTT-menable.
All hyperbolic groups are also a-TTT-menable.
(See Section 7.$\mathrm{E}_1$ in \cite{gromov}. More explicitly,
$\cc(g):=q[1,g]$ in the notation of Theorem 10 in \cite{mineyev}
is a proper quasi-cocycle.)
 From the above discussion, we have the following consequence.

\begin{thmA}\label{amenable}
Let $G$ and $G'$ be locally compact groups
such that $G$ with property {\TTT} and $G'$ a-TTT-menable.
Then, every quasi-homomorphism from $G$ into $G'$ has
a relatively compact image.
\end{thmA}

We will prove that the inclusion of an abelian group $A$ into
a semidirect product group $G_0\ltimes A$ has Kazhdan's relative property (T)
if and only if it has relative property {\TTT}.
It follows that the group $\SL_{n\geq3}(\IR)$ has property {\TTT}.
We then make an extra effort to prove that property {\TTT} is
inherited to lattices (under a certain condition).

\begin{thmA}\label{sl}
For any local field $\IK$ and $n\geq3$,
the group $\SL_n(\IK)$ and its lattices have property \TTT.
\end{thmA}

Unlike the case of property (T), it is not so straightforward to prove
that property {\TTT} is inherited to lattices.
In the case of property (TT), Burger and Monod (\cite{burger-monod,burger-monod2})
use bounded cohomology machinery.
For property {\TTT}, we replace their cohomological theorem
(injectivity of the $L^2$-induction map) with
the following theorem about a length function on a measured transformation groupoid.

\begin{thmA}\label{length}
Let $G\curvearrowright X$ be a measure-preserving action
of a locally compact group $G$ on a standard probability space $X$,
and $\el\colon X\times G\to\IR_{\geq0}$ be a measurable function such that
\[
\el(x,gh) \le \el(x,g) + \el(g^{-1}x,h)
\]
for a.e.\ $(x,g,h) \in X\times G\times G$.
If
\[
\esup_{g\in G}\int_X \el(x,g)\,dx<+\infty,
\]
then there exists $h\in L^1(X)$ such that
\[
\el(x,g)\le h(x)+h(g^{-1}x)
\]
almost everywhere.
\end{thmA}

When one views a length function $\ell$ as a kind of cocycle,
this theorem says that if a ``cocycle'' is bounded in a certain sense,
then it is dominated by a ``coboundary.''
A more precise version of this theorem is given as Theorem~\ref{thm:oid}.

\subsection*{Acknowledgment}
The author would like to thank Professor M.~Burger and
Professor N.~Monod for useful conversations on this work.
In particular, extending Theorem~\ref{amenable} to
cover measurable quasi-cocycles (Corollary~\ref{cor:ug})
was suggested by Monod.
This research was carried out while the author was staying at
the Hausdorff Research Institute for Mathematics (HIM)
for Trimester Program on ``Rigidity.''
He gratefully acknowledges the kind hospitality and
stimulating environment provided by HIM and the program organizers.
\section{Preliminaries on abstract harmonic analysis}\label{sec:prelim}
In this section, we collect useful facts from abstract harmonic analysis.
We refer the reader to \cite{bhv,brown-ozawa,cowling-haagerup,hewitt-ross,pisier}
for more information.
Let $G$ be a locally compact group and
denote by $\lambda$ the left regular representation of $G$ on $L^2(G)$.
We extend $\lambda$ to the Banach algebra $L^1(G)$ by
\[
(\lambda(f)\zeta)(x)=\int_G f(g)\zeta(g^{-1}x)\,dg=(f*\zeta)(x).
\]
The reduced group C$^*$-algebra $C^*_r(G)$ is defined as
the norm-closure of $\lambda(L^1(G))$ in $\IB(L^2(G))$.
When $G=A$ is abelian, the Fourier transform $L^2(A)\cong L^2(\widehat{A})$
implements a canonical $*$-isomorphism between $C^*_r(A)$ and
the C$^*$-algebra $C_0(\widehat{A})$ of all continuous functions
on the Pontrjagin dual $\widehat{A}$ that vanish at infinity.

Recall that a kernel $\theta\colon G\times G\to \IC$ is said to be
positive definite if $\sum_{i,j}\theta(x_i,x_j)\xi_i\overline{\xi_j}\geq0$
for any $n$ and $x_1,\ldots,x_n \in G$, $\xi_1,\ldots,\xi_n\in\IC$. 
It is well-known that $\theta$ is positive definite if and only if 
there is a map $F$ from $G$ into a Hilbert space $\cK$ such that 
$\theta(x,y)=\ip{F(x),F(y)}$.
A positive definite kernel $\theta$ is said to be \emph{normalized}
if $\theta(x,x)=1$ for all $x\in G$.
One can normalize $\theta(x,y)=\ip{F(x),F(y)}$
by replacing it with $\ip{\frac{F(x)}{\|F(x)\|},\frac{F(y)}{\|F(y)\|}}$.
We note that if $\theta$ is a normalized positive definite kernel, then
$\theta(x,y)\approx1$ implies $\theta(x,z)\approx\theta(y,z)$ uniformly for $z\in G$,
or more precisely
\[
|\theta(x,z)-\theta(y,z)|\le\|F(x)-F(y)\|\,\|F(z)\|\le \sqrt{2}|1-\theta(x,y)|^{1/2}.
\]
If $\theta$ is a bounded measurable positive definite kernel,
then $F$ as above is weakly measurable, i.e.,
$x\mapsto \ip{F(x),v}$ is measurable for every $v\in\cK$, and
\[
\int\theta(x,y)\xi(x)\overline{\xi(y)}\,dx\,dy=\ip{\int\xi(x)F(x)\,dx,\int\xi(y)F(y)\,dy}\geq0
\]
for every $\xi\in L^1(G)$, where $\int\xi(x)F(x)\,dx$ is the unique element in $\cK$ such that
$\ip{\int\xi(x)F(x)\,dx,v}=\int\xi(x)\ip{F(x),v}\,dx$ for every $v\in\cK$.
With this in mind, we say a kernel $\theta\in L^\infty(G\times G)$
is \emph{(essentially) positive definite} if
$\int\theta(x,y)\xi(x)\overline{\xi(y)}\,dx\,dy\geq0$ for every $\xi\in L^1(G)$.
A kernel $\theta\in L^\infty(G\times G)$ is positive definite if and only if
there is a measurable function $P$ from $G$ into a separable
Hilbert space $\cH$ such that $\theta(x,y)=\ip{P(x),P(y)}$ a.e.
Moreover, if $\theta$ is continuous in addition, then $P$ is continuous and
the previous equality holds everywhere.
Indeed, if $\theta$ is positive definite, then there are a 
Hilbert space $\cH$ and a bounded linear map $T\colon L^1(G)\to\cH$
such that $\ip{T\xi,T\eta}=\int\theta(x,y)\xi(x)\overline{\eta(y)}\,dx\,dy$,
because the right hand side defines a semi-inner product on $L^1(G)$.
But, every bounded linear map $T\colon L^1(G)\to\cH$ is represented by
$P\in L^\infty(G,\cH)$ in such a way that $T\xi=\int \xi(x)P(x)\,dx$.
It follows that $\|P\|=\|T\|=\|\theta\|_\infty^{1/2}$ and
$\theta(x,y)=\ip{P(x),P(y)}$ a.e.
When $\theta$ is continuous, $P$ can be taken continuous.
%We note that for any $P \in L^\infty(G,\cH)$ with $\cH$ separable
%such that $\theta(x,y)=\ip{P(x),P(y)}$ a.e.\
%and for any non-negligible measurable subset $E\subset G$, one has
%\[
%\|\theta|_{E\times E}\|_{L^\infty(E\times E)}=\|P|_E\|_{L^\infty(E,\cH)}^2,
%\]
%by Lusin's theorem.
Next, we consider not necessarily positive definite $\theta \in L^\infty(G\times G)$
and define the $\cb$-norm of $\theta$ by
\[
\|\theta\|_{\cb}
=\inf\{ \|P\|\,\|Q\| :
 P,Q\in L^\infty(G,\cH),\ \theta(x,y)=\ip{P(x),Q(y)}\mbox{ a.e.}\},
\]
where the infimum over empty set is equal to $\infty$.
This norm is also described via positive definite kernels.
Let $G^{(2)}=G\sqcup G$ be the disjoint union of two copies of $G$, and
$\iota_{1,2}$ be the inclusion of $G\times G$ into the $(1,2)$-component of
$G^{(2)}\times G^{(2)}$. Then, one has
\[
\|\theta\|_{\cb}=\inf\{ \|\tilde{\theta}\|_\infty :
 \tilde{\theta}\in L^\infty(G^{(2)}\times G^{(2)})
 \mbox{ pos.\ def., }\tilde{\theta}\circ\iota_{1,2}=\theta\}.
\]
%In particular, we may assume that the Hilbert spaces $\cH$
%appearing in the definition of cb-norm are separable.
Moreover, $\|\theta\|_{\cb}$ coincides with the operator norm
viewed as a Schur multiplier on $\IB(L^2(G))$.
See \cite{haagerup} or Section 3.2 of \cite{spronk} for more information.

Now recall $G$ is a group and define $g\cdot\theta$ for
$g\in G$ and $\theta\in L^\infty(G\times G)$ by
\[
(g\cdot\theta)(x,y)=\theta(g^{-1}x,g^{-1}y),
\]
and suppose that $\theta$ is continuous,
left-invariant, say $\theta(g,h)=\p(g^{-1}h)$ for a continuous function $\p$,
and has finite $\cb$-norm.
Then, the function $\p$ is a Herz--Schur multiplier.
Namely, $m_\theta\colon L^1(G)\ni f\mapsto \p f\in L^1(G)$ extends to the
reduced group C$^*$-algebra $C^*_r(G)$,
which satisfies $\|m_\theta\|\le\|\theta\|_{\cb}$.
Indeed, given an expression $\theta(g,h)=\ip{P(g),Q(h)}$, we define the operator
$V_P\colon L^2(G)\to L^2(G,\cH)$ by
$(V_P\zeta)(x)=\zeta(x)P(x^{-1})$. It is clear that $\|V_P\|\le\|P\|$.
Likewise for $V_Q$. We then observe that
$\lambda(\p f)=V_Q^*(\lambda(f)\otimes1_{\cH})V_P$.
(In fact, $\|\theta\|_{\cb}$ coincides with the cb-norm of $m_\theta$.)
Suppose moreover that $G=A$ is abelian.
We denote by $0_A$ the unit character on $A$ and view it as a character
on $C^*_r(A)\cong C_0(\widehat{A})$.
Since the liner functional $0_A\circ m_\theta$ on $C^*_r(A)\cong C_0(\widehat{A})$
is bounded, it is given by a finite complex Borel measure $\mu_\theta$
on $\widehat{A}$ by the Riesz--Markov representation theorem.
One has $\|\mu_\theta\|\le\|\theta\|_{\cb}$ (actually they coincide),
and $\int_{\widehat{A}} \widehat{a}\,d\mu_\theta=\p(a)$
for all $a\in A$.
Also note that $\mu_\theta$ is positive if and only if $\theta$ is positive definite.
%In fact, one may take $\p$ to be continuous and,
%in such a case, the above equality holds for all $g\in G$.
%Since we need this fact only for the positive case, we assume that
%$\theta$ is positive definite.
%(For the general case, consider the Hahn decomposition of $\theta$.)
%By the Segal--von Neumann theorem (see, e.g., 32.12 in \cite{hewitt-ross}),
%the positive type function $\p$ decomposes as a sum $\p=\p_c+\p_s$
%of two positive type functions with $\p_c$ continuous and $\p_s=0$ a.e.
%Let $f_n\in L^1(G)$ be non-negative functions such that $\int f_n(x)\,dx=1$
%and $f_n\to\delta_g$. Then, $\hat{f}_n\to\hat{g}$ and
%the Lebesgue Dominated Convergence Theorem implies that
%\[
%\int_{\widehat{G}} \widehat{g}\,d\mu_\theta
%=\lim_{n\to\infty}\int_{\widehat{G}} \hat{f}_n\,d\mu_\theta
%=\int_G \p(x) f_n(x)\,dx=\p_c(g).
%\]

We explain how to take a (generalized) limit in dual Banach spaces.
Fix a positive linear functional $\Lim\colon\ell_\infty(\IN)\to\IC$
which extends the usual limit on the convergent sequences.
Let $V=(V_*)^*$ be a dual Banach space and $(v_n)$ be a
bounded sequence in $V$.
Then, the limit $\Lim_n v_n$ is defined to be
the unique element in $V$ that satisfies
\[
\ip{\Lim_n v_n,\xi}=\Lim_n\ip{v_n,\xi}
\]
for $\xi\in V_*$, where $\ip{\,\cdot\,,\,\cdot\,}$ denotes the
duality coupling between $V$ and its predual $V_*$.
One has $\|\Lim_n v_n\|\le \Lim_n \|v_n\|$.
If $T$ is an operator from $V$ into
another dual Banach space $W=(W_*)^*$ that is weak$^*$-continuous,
i.e., $T^*\xi:=\xi\circ T$ belongs to $V_*$ for every $\xi\in W_*$,
then $T (\Lim_n v_n) = \Lim_n T(v_n)$.
Indeed,
\[
\ip{T\Lim_n v_n,\xi}=\ip{\Lim_n v_n,T^*\xi}
=\Lim_n\ip{v_n,T^*\xi}
=\Lim_n\ip{Tv_n,\xi}=\ip{\Lim_n Tv_n,\xi}
\]
for every $\xi\in W_*$.
Now suppose $Y$ is a $\sigma$-finite measure space and
$(f_n)$ is a real bounded sequence in $L^\infty(Y)=(L^1(Y))^*$.
Then, although $\Lim_n f_n$ may not coincide with the pointwise
limit $\Lim_n f_n(y)$, one still has
\[
\liminf_n f_n(y)\le (\Lim_n f_n)(y) \le \limsup_n f_n(y)
\]
for a.e.\ $y\in Y$. Indeed, for any $B\subset Y$ with finite measure,
one has
\[
\int_B \liminf_n f_n\,dy \le \liminf_n \int_B f_n\,dy
\le \Lim_n \int_B f_n\,dy = \int_B \Lim_n f_n\,dy.
\]

We observe that if $\theta_n\in L^\infty(G\times G)$ have uniformly bounded cb-norm,
then $\Lim_n\theta_n\in L^\infty(G\times G)$ has cb-norm at most $\Lim_n\|\theta_n\|_{\cb}$.
Indeed, this follows from the fact that $\Lim_n\theta_n$ is
positive definite if $\theta_n$ are.
(In fact, the Banach space $V_2(G)$ of those kernels
that have finite $\cb$-norm is a dual Banach space and
the natural map from $V_2(G)$ into $L^\infty(G\times G)$
is a weak$^*$-continuous injection.)
We finally note that for a continuous kernel
$\theta$ on $G$ and a closed subgroup $A\le G$,
the restriction of $\theta$ to the subgroup $A$ satisfies
$\| \theta|_{A\times A} \|_{\cb} \le \|\theta\|_{\cb}$.
Indeed, let $\theta(x,y)=\ip{P(x),Q(y)}$ for a.e.\ $x,y\in G$.
Take an approximate unit $(f_n)$ of $L^1(G)$ and let
$P_n(x)=\int_G f_n(g)P(g^{-1}x)\,dg$, and likewise for $Q_n$.
Then, $P_n$ and $Q_n$ are continuous and the continuous kernel
$\theta_n(x,y)=\ip{P_n(x),Q_n(y)}$ has cb-norm at most $\|P\|\,\|Q\|$.
The same thing holds for $\theta_n|_{A\times A}$.
Taking the limit, we are done.
\section{Relationship to other formulations}\label{sec:tpq}
Property (T) has several equivalent characterizations
(see Section 2.12 in \cite{bhv} or Theorem 12.1.7 in \cite{brown-ozawa}).
In this section, we pursue analogous characterizations.

\begin{defn}
Let $A\le G$ be a subgroup of a locally compact group $G$.
We say the pair $(G,A)$ has \emph{relative property \TTT}
if every wq-cocycle on $G$ is bounded on $A$.
We say the pair $(G,A)$ has \emph{relative property {\Tp}} or {\Tq} respectively,
if for every $\e>0$, there exist a compact subset $K\subset G$ and $\delta>0$
(we will take $\delta<\e$ for granted) satisfying the following condition.
\begin{itemize}
\item[\Tp:]
If $\theta\colon G\times G\to \IC$ is a (normalized) Borel positive definite kernel
such that
\[
\sup_{g\in G}\| g\cdot\theta - \theta \|_{\cb} < \delta
\ \mbox{ and }\
\sup_{g^{-1}h\in K} | \theta(g,h) - 1 | < \delta,
\]
then one has
\[
\sup_{x,y\in A} | \theta(x,y) - 1 | < \e.
\]
%\item[\Tp:]
%If $\theta\in L^\infty(G\times G)$ is a (normalized) positive definite kernel
%such that
%\[
%\sup_{g\in G}\| g\cdot\theta - \theta \|_{\cb} < \delta
%\ \mbox{ and }\
%\esup_{(g,h)\in G^2,\, g^{-1}h\in K} | \theta(g,h) - 1 | < \delta,
%\]
%then for the constant function $\mathbf{1}_A$ on $A\times A$, one has
%\[
%\|\theta|_A-\mathbf{1}_A\|_{\cb}<\e.
%\esup_{(x,y)\in A^2}| \theta(x,y) - 1 | < \e.
%\]
\item[\Tq:]
If $\pi\colon G\to\cU(\cH)$ is a Borel map and $\xi\in\cH$ is a unit vector
such that
\[
\sup_{g,h\in G} \|\pi(gh)\xi-\pi(g)\pi(h)\xi\|<\delta
\ \mbox{ and }\
\sup_{g\in K}\|\pi(g)\xi-\xi\|<\delta,
\]
then one has
\[
\sup_{x\in A}\|\pi(x)\xi-\xi\|<\e.
\]
\end{itemize}
When $(G,G)$ has relative property {\Tp} or {\Tq},
we simply say $G$ has \emph{property {\Tp}} or {\Tq} respectively.
\end{defn}

We remark that the term ``Borel'' in the definition of property {\Tp}
can be replaced with ``continuous.''
Indeed, one can replace any Borel positive definite kernel $\theta$
with a continuous $\tilde{\theta}$, defined by
\[
\tilde{\theta}(x,y)=\frac{1}{|K|^2}\int_{K\times K}\theta(xk,yk')\,dk\,dk',
\]
where $K$ is a compact non-negligible subset.
Note that $\tilde{\theta}$ is uniformly close to the original $\theta$
if $\theta(x,xk)\approx1$ for all $(x,k)\in G\times K$.

Property {\Tq} is suited for the study of rigidity phenomena
in the setting of $\e$-representations (see \cite{kazhdan,bot}).
We will prove that $\Tp\Rightarrow\TTT\Rightarrow\Tq$, but
it is unclear whether they are all equivalent.
Relative property {\Tq} implies relative property (T).
This fact is not hard to show when $A$ is normal.
For the general case, see \cite{jolissaint}.

\begin{thm}\label{thm:tpq}
For a pair $(G,A)$ as above, one has
\[
\mbox{\normalfont rel.\ property \Tp} \Rightarrow
\mbox{\normalfont rel.\ property \TTT} \Rightarrow
\mbox{\normalfont rel.\ property \Tq}.
\]
\end{thm}
\begin{proof}
We only prove $\Tp\Rightarrow\TTT$, and omit the proof of $\TTT\Rightarrow\Tq$
because it is virtually same as the classical one
(see Proposition 2.4.5 in \cite{bhv}).
Note that we are assuming $G$ is second countable.

Let $\e=1/2$ and take $(K,\delta)$ which satisfies condition {\Tp}.
We may assume that $K$ is symmetric and contains the unit.
Let $\cc\colon G\to\cH$ be a wq-cocycle.
Considering realification, we may assume that $\cc$ is real
and $\pi$ is orthogonal.
By scaling $\cc$, we may further assume that
\[
\sup_{g,h\in G}\|\cc(gh)-\big(\cc(g)+\pi(g)\cc(h)\big)\|\le\delta_0
\ \mbox{ and }\
\sup_{g\in K}\|\cc(g)\|\le\delta_0,
\]
where $\delta_0>0$ is a sufficiently small number which will be chosen later.
We consider the full Fock Hilbert space $\cF=\bigoplus_{n=0}^\infty\cH^{\otimes n}$,
where $\cH^{\otimes 0}=\IR$, and the exponential map $\EXP\colon\cH\to\cF$ given by
\[
\EXP(\xi)=
1\oplus\frac{\xi}{\sqrt{1!}}\oplus \frac{\xi\otimes\xi}{\sqrt{2!}}\oplus
\frac{\xi\otimes\xi\otimes\xi}{\sqrt{3!}}\oplus\cdots.
\]
We define $E\colon \cH\to\cF$ by $E(\xi)=\exp(-\|\xi\|^2)\EXP(\sqrt{2}\xi)$.
It follows
\[
\ip{E(\xi),E(\eta)}=\exp(-\|\xi-\eta\|^2)
\]
for all $\xi,\eta\in\cH$.
In particular, $E$ is a continuous map into the unit sphere of $\cF$.
Consider the normalized Borel positive definite kernel
\[
\theta(x,y)=\ip{E(\cc(x)),E(\cc(y))}=\exp(-\|\cc(x)-\cc(y)\|^2).
\]
Since
\[
\theta(x,y)=\ip{E\bigl(\cc(g^{-1})+\pi(g^{-1})\cc(x)\bigr),E\bigl(\cc(g^{-1})+\pi(g^{-1})\cc(y)\bigr)},
\]
one has
\[
\|g\cdot\theta-\theta\|_{\cb}\le 2\sup_{x\in G}\|E(\cc(g^{-1}x))-E\bigl(\cc(g^{-1})+\pi(g^{-1})\cc(x)\bigr)\|
<\delta_1
\]
for all $g\in G$, where $\delta_1=2(2-2\exp(-\delta_0^2))^{\oh}$.
Also,
\[
|\theta(g,h)-1|=1-\exp(-\|\cc(g)-\cc(h)\|^2) < 1-\exp(-4\delta_0^2)
\]
for all $(g,h)\in G^2$ that satisfy $g^{-1}h\in K$.
Thus, if $\delta_0>0$ was chosen sufficiently small,
then property {\Tp} implies
\[
1-\exp(-\|\cc(x)-\cc(1)\|^2)=| \theta(x,1) - 1 | < \e = 1/2.
\]
for all $x\in A$.
This means that $\cc$ is bounded on $A$.
\end{proof}

\begin{cor}\label{cor:ug}
Let $G\curvearrowright X$ be a measure-preserving action on a standard probability space $X$
and $\beta\colon X\times G\to G'$ be a measurable map having the following property:
For every compact subset $K\subset G$ and a.e.\ $x\in X$,
the set
\[
\{\beta(x,g) : g\in K\}\cup\{ \beta(x,gh)^{-1}\beta(x,g)\beta(g^{-1}x,h) : g,h\in G\}
\]
is relatively compact in $G'$.
If $G$ has property {\Tq} and $G'$ is a-T-menable, then
there exists a sequence $X_1\subset X_2\subset\cdots\subset X$ such that
$\bigcup X_n$ is co-null in $X$ and
\[
\{ \beta(x,g) : x\in X_n,\, g\in G\mbox{ such that }g^{-1}x\in X_n\}
\]
are relatively compact in $G'$ for all $n$.
\end{cor}
The proof of Corollary~\ref{cor:ug} will be given at the end of Section~\ref{sec:length}.
\section{Property {\TTT} for $SL_n(\IK)$}\label{sec:slnk}
In this section, we prove property {\Tp} for $SL_n(\IK)$.
We follow a well-established line of the proof for property (T)
(see Section 1.4 in \cite{bhv}),
also employing ideas from \cite{burger,shalom:tams,shalom:bdd}.

Let $G$ be a locally compact group and $A\le G$ be an abelian closed
normal subgroup.
Then, $G$ acts on the Pontrjagin dual $\widehat{A}$ by the dual action
of the conjugate action.
This action induces an isometric action of $G$ on the Banach space
$\cM(\widehat{A})$ of finite regular Borel measures on $\widehat{A}$.

\begin{prop}\label{prop:nac}
Let $G=G_0\ltimes A$ be the semidirect product of
a locally compact abelian group $A$ by
a continuous action of a locally compact group $G_0$.
Then, the following are equivalent.
\begin{enumerate}
\item\label{item:T} The pair $(G,A)$ has relative property $\mathrm{(T)}$.
\item\label{item:TTT} The pair $(G,A)$ has relative property {\Tp} (resp.\ {\TTT}, {\Tq}).
\item\label{item:m} For every $\e>0$, there exist a compact subset
$K\subset G_0\cup A$ and $\delta>0$ which have the following property:
If $\mu$ is a probability measure on $\widehat{A}$ such that
$\|g\cdot\mu-\mu\|\le\delta$ for $g\in K$ and
$|1-\int_{\widehat{A}} \widehat{a}\,d\mu|<\delta$ for $a\in K\cap A$,
then $|1-\int_{\widehat{A}} \widehat{x}\,d\mu|<\e$ for all $x\in A$.
\end{enumerate}
\end{prop}
\begin{proof}
The implication $(\ref{item:TTT})\Rightarrow(\ref{item:T})$ is explained in Section~\ref{sec:tpq},
and $(\ref{item:T})\Rightarrow(\ref{item:m})$ is proved in \cite{ct,ioana}.
Now, we assume $(\ref{item:m})$ and prove that the pair $(G,A)$ has relative property {\Tp}.
Let $\e>0$ be given and take $(K,\delta_0)$ which satisfies condition~$(\ref{item:m})$.
We assume that $K$ is symmetric and contains a neighborhood of the unit.
Let $\theta\colon G\times G\to \IC$ be a normalized continuous positive definite kernel
such that
\[
\sup_{x\in G}\| x\cdot\theta - \theta \|_{\cb} < \delta
\ \mbox{ and }\
\sup_{g^{-1}h\in K} | \theta(g,h) - 1 | < \delta,
\]
where $\delta>0$ is a sufficiently small number which will be chosen later.
We will prove that $\sup_{a,b\in A}|\theta(a,b)-1|<2\e$ for every $a,b\in A$.
Express $\theta$ as $\theta(x,y)=\ip{P(x),P(y)}$.
Then, one has
\[
\| P(x) - P(xg) \|^2
=2-2\Re\theta(x,xg) <2\delta
\]
for all $x\in G$ and $g\in K$.
In particular,
\[
%\sup_{x,y\in G}|\theta(x,y)-\theta(xg,yh)|< 3\delta^{\oh}
%\mbox{ and }
\sup_{g,h\in K}\|\theta(\,\cdot\,g,\,\cdot\,h)-\theta(\,\cdot\,,\,\cdot\,)\|_{\cb}
\le 3\delta^{\oh}.
\]

We will average $\theta$ over $A$ to obtain an $A$-invariant kernel.
A similar idea is used in \cite{mimura}.
For every $a\in A$ and $g\in K$, define a kernel $\theta_a^g$ on $A$ by
\[
(\theta_a^g)(x,y)=\frac{1}{|K|^2}\int_{K\times K}\theta(agxg^{-1}k,agyg^{-1}k')\,dk\,dk'.
\]
Then, the family $\{\theta_a^g : a\in A,\, g\in K\}$ is (right) equicontinuous
and satisfies $\| \theta_a^g - \theta\|_{\cb}\le \delta+6\delta^{\oh}=:\delta_1$
(see the last paragraph of Section~\ref{sec:prelim}).
Take a F{\o}lner sequence $L_n\subset A$ and
consider the kernel
\[
\theta_n^g(x,y)
=\frac{1}{|L_n|} \int_{L_n} \theta_a^g(x,y)\,da,
\]
where the integration is with respect to the Haar measure of $A$.
We fix a free ultrafilter on $\IN$ and denote the associated ultralimit by $\Lim_n$.
Then, $\tilde{\theta}^g(x,y)=\Lim_n \theta_n^g(x,y)$ is an $A$-invariant
continuous positive definite kernel such that
$\|\tilde{\theta}^g-\theta\|_{\cb}\le\delta_1$
for $g\in K$.

Now, let $\mu_{\tilde{\theta}^1}$ be the measure associated with
the $A$-invariant continuous positive definite kernel $\tilde{\theta}^1$.
One has $1\geq\mu_{\tilde{\theta}^1}(\widehat{A})=\tilde{\theta}^1(1,1)\geq1-\delta_1$,
\[
\sup_{g\in K}\|g\cdot\mu_{\tilde{\theta}^1}-\mu_{\tilde{\theta}^1}\|
\le\sup_{g\in K}\|\tilde{\theta}^g-\tilde{\theta}^1\|_{\cb}\le2\delta_1
\]
and
\[
\sup_{a\in K\cap A}|1-\int_{\widehat{A}}\widehat{a}\,d\mu_{\tilde{\theta}^1}|
=\sup_{a\in K\cap A}|1-\tilde{\theta}^1(a,1)|\le 2\delta_1.
\]
Thus, if $\delta>0$ was chosen sufficiently small,
then condition~(\ref{item:m}) implies
\[
\sup_{x,y\in A}|\tilde{\theta}^1(x,y)-1|
=\sup_{x,y\in A}|1-\int_{\widehat{A}}\widehat{x^{-1}y}\,d\mu_{\tilde{\theta}^1}|\le\e.
\]
It follows that
\[
\sup_{x,y\in A}|\theta(x,y)-1|\le\e+\delta_1.
\]
This completes the proof.
\end{proof}

We remark that one can prove in a similar manner
the following strengthening of Theorem 5.5 in \cite{shalom:tams}.
\begin{prop}
Let $G$ be a locally compact group and $A$ be an abelian closed
normal subgroup. Assume that there is no $G$-invariant finitely additive
probability measure defined on the Borel subsets of $\widehat{A}-\{0_A\}$.
Then, the pair $(G,A)$ has relative property \Tp.
\end{prop}

Now, we prove one half of Theorem~\ref{sl}.
Let $R$ be a unital commutative ring.
We recall that an elementary matrix means
an element in $\SL_n(R)$ of the form $E_{i,j}(r)=I+ re_{i,j}$ for some $i\neq j$ and $r\in R$,
and $\EL_n(R)$ denotes the subgroup of $\SL_n(R)$ generated by elementary matrices.
The group $\EL_n(R)$ is \emph{boundedly elementary generated} if
there is a number $l=l(n,R)$ such that every element in $\EL_n(R)$ can be written
as a product of at most $l$ elementary matrices.
(See \cite{shalom:bdd} or Chapter 4 in \cite{bhv}.)
Thanks to the Gaussian elimination process, for any field $R$,
one has $\EL_n(R)=\SL_n(R)$ and it is boundedly elementary generated.
\begin{thm}\label{thm:sltp}
For any local field $\IK$ and $n\geq3$, the group $\SL_n(\IK)$ has property \Tp.
For any finitely generated unital commutative ring $R$ and $n\geq 3$ such that
$\EL_n(R)$ is boundedly elementary generated, the discrete group $\EL_n(R)$ has property \Tp.
\end{thm}

\begin{proof}
Let $R=\IK$ or a finitely generated unital commutative ring, and $G=\EL_n(R)$.
The pair $R^2\triangleleft\EL_2(R)\ltimes R^2$ has relative property (T)
by Corollary 1.4.13 in \cite{bhv} and by \cite{shalom:bdd}.
Thus it has relative property {\Tp} as well by Theorem~\ref{thm:tpq}.
Let $\e>0$ be arbitrary and take $(K_0,\delta)$ which satisfies condition {\Tp}.
For each pair $i\neq j$, there is an embedding
$\sigma_{i,j}\colon\EL_2(R)\ltimes R^2\to\EL_n(R)$
such that $E_{i,j}(R)\subset\sigma_{i,j}(R^2)$.
Let $K=\bigcup \sigma_{i,j}(K)\subset \EL_n(R)$.
Suppose that $\theta$ is a normalized continuous positive definite kernel
on $G$ such that
\[
\sup_{g\in G}\| g\cdot\theta - \theta \|_{\cb} < \delta
\ \mbox{ and }\
\sup_{g^{-1}h\in K} | \theta(g,h) - 1 | < \delta.
\]
Then, by relative property {\Tp}, one has
\[
\sup_{s\in E_{i,j}(R)}| \theta(s,1) - 1 | < \e.
\]
It follows (see Section~\ref{sec:prelim}) that
\[
|\theta(gs,1)-\theta(g,1)|^2\le 2|1-\theta(gs,g)|<2(\delta+\e)<4\e
\]
for all $g\in G$ and $s\in E_{i,j}(R)$.
By bounded generation property, this implies $|1-\theta(x,1)|<2l\e^{\oh}$ for all $x\in G$.
\end{proof}

\begin{rem}
The group $\SL_n(\IK)$ actually have st.pr.{\Tp} in the sense of \cite{shalom:tams}.
Namely, one can take $K$ to be finite rather than compact, and
any wq-cocycle on $\SL_n(\IK)$, which is not assumed locally bounded,
is bounded. For the proof, mimic \cite{shalom:tams}, or use \cite{ct}.

There is a variant of Mautner's lemma:
Let $\theta$ be a normalized continuous positive definite kernel on $G$
such that $\| g\cdot\theta - \theta \|_{\cb} < \e$ for all $g\in G$,
and $x,y\in G$ be such that
$|\theta(y,1) -1 |<\e$ and $|\theta(y^{-1}xy,1) -1 |<\e$.
Then, $|\theta(x,1) - 1|<2\e+4\e^{\oh}$.
\end{rem}
\section{Induction and length-like functions}\label{sec:length}
In this section, we prove the remaining half of Theorem~\ref{sl}.
The proof will involve a general discussion about length-like functions
on measured-groupoids.

\begin{thm}\label{thm:lattice}
Let $G$ be a locally compact group and $\G\le G$ be a lattice.
Then, $G$ has property {\Tp} if and only if $\G$ has property {\Tp}.
\end{thm}

Now Theorem~\ref{sl} follows from
Theorems \ref{thm:sltp}, \ref{thm:lattice} and \ref{thm:tpq}.
We remark that property {\Tp} is moreover a measure-equivalence invariant,
and the same thing holds for property {\Tq}.
On the other hand, it is unclear whether property {\TTT} is inherited
to a lattice unless the lattice is cocompact, because one needs a certain
integrability condition to induce wq-cocycles.
We do not prove these facts, because we will not (probably ever) need them.
For the proof of Theorem~\ref{thm:lattice}, we use a random walk technique,
in particular double ergodicity of a Poisson boundary, which is also a key
ingredient in the proof of the fact that property (TT) is inherited to
lattices (\cite{burger-monod,burger-monod2}).
Thus, we fix a symmetric non-degenerate probability measure $\mu$ on $G$,
which is absolutely continuous with respect to the Haar measure.
Such a measure $\mu$ always exists (because we are assuming that $G$ is second countable).
Let $V$ be a coefficient $G$-module
(i.e., $V$ is a dual Banach space on which $G$ acts by weak$^*$-continuous isometries)
and $f\in L^\infty(G,V)$.
We define
\[
(\mu*f)(g)=\int_G s\cdot f(s^{-1}g)\,d\mu(s)
\ \mbox{ and }\
(f*\mu)(g)=\int_G f(gt^{-1})\,d\mu(t).
\]
The following is an incarnation of double ergodicity of a Poisson boundary (\cite{kaimanovich}).
It is also considered as a noncommutative Choquet--Deny theorem with coefficients (cf.\ Theorem 1 in \cite{willis}).
\begin{lem}\label{lem:ncd}
Assume $V$ is a separable coefficient $G$-module and
$f\in L^\infty(G, V)$ is such that $\mu*f=f=f*\mu$.
Then, there exists a $G$-invariant vector $v_0$ in $V$ such that
$f=v_0$ almost everywhere.
\end{lem}
\begin{proof}
For $m,n\in\IN$, define $F_{m,n}\colon (G,\mu)^{\IZ}\to V$ by
\[
F_{m,n}((g_k)_{k\in\IZ})=g_{0}^{-1}\cdots g_{-m}^{-1}\cdot f(g_{-m}\cdots g_0g_1\cdots g_n).
\]
By the martingale convergence theorem, $(F_{m,n})$ converges
a.e.\ as $m,n\to\infty$.
The limit function $F$ satisfies $F((g_{k+1})_{k\in\IZ})=g_1^{-1}F((g_k)_{k\in\IZ})$,
and hence is constant by Theorem 6 in \cite{kaimanovich}, say $F=v_0$.
Note that $v_0$ is a $G$-invariant vector.
Since $(F_{m,n})$ is uniformly bounded,
for every measurable subsets $B_1,\ldots,B_l\subset G$, one has
\begin{align*}
\int_G f(g)\,d\big((\mu|_{B_1})*\cdots*(\mu|_{B_l})\big)(g)
&=\int_{G^{\IZ}} F_{m,n}(\pmb{g})\chi_{B_1}(g_1)\cdots\chi_{B_l}(g_l)\,d\mu^{\infty}(\pmb{g})\\
&\to\mu(B_1)\cdots\mu(B_l)\,v_0\quad\mbox{ as $m,n\to\infty$}.
\end{align*}
This implies that $f=v_0$ almost everywhere.
\end{proof}

\begin{thm}\label{thm:oid}
Let $G\curvearrowright X$ be a measure-preserving action
on a standard probability space $X$,
$C\geq1$ and $\el\colon X\times G\to\IR_{\geq0}$
be a measurable function such that
\[
\el(x,gh) \le C( \el(x,g) + \el(g^{-1}x,h) )
\]
for a.e.\ $(x,g,h) \in X\times G\times G$.
Assume that
\[
D:=\limsup_{n\to\infty}\int_G\int_X \el(x,g)\,dx\,d\mu^{*n}(g)<+\infty.
\]
Then, there exists $h\in L^1(X)$ such that $\|h\|_1\le4C^4D$ and
\[
\el(x,g)\le h(x)+h(g^{-1}x)
\]
almost everywhere.
\end{thm}
\begin{proof}
Take $R>0$ arbitrary. Let $\el_R=\min(\el,R)$ and consider
\[
f_R:=\Lim_{m,n}\,\mu^{*m}*\el_R*\mu^{*n} \in L^\infty(X\times G)\cong L^\infty(G,L^\infty(X)),
\]
where the limit is taken along an invariant mean $\Lim_{m,n}$ on $\IN^2$
with respect to the weak$^*$-topology on $L^\infty(X\times G)$.
(See the discussion in Section~\ref{sec:prelim}.)
Since the convolutions by $\mu$ are weak$^*$-continuous operators
on $L^\infty(X\times G)$, one has $\mu*f_R=f_R=f_R*\mu$.
Since $L^\infty(X)$ is contained in the separable coefficient $G$-module $L^2(X)$,
Lemma~\ref{lem:ncd} implies that $f_R$ belongs to $L^\infty(X)$ and is $G$-invariant.
We note that
\[
(\mu^{*m}*\el_R*\mu^{*n})(x,g)=\int_{G^2} \el_R(sx,sgt)\,d\mu^{*m}(s)\,d\mu^{*n}(t).
\]
Choose a subset $A\subset\{ g\in G : \int_X \el(x,g)\,dx \le 2D\}$ of measure $1$
(it is not difficult to see that the latter set has infinite measure).
One has
\begin{align*}
\|f_R & \|_{L^1(X)} = \int_{X\times A} f_R(x,g)\,d(x,g)\\
&=\Lim_{m,n} \int_{X\times A} (\mu^{*m}*\el_R*\mu^{*n})(x,g)\,d(x,g)\\
&\le \Lim_{m,n}C^2\int_{X\times A}\int_{G^2}\left(\begin{array}{l}
  \el(sx,s)+\el(x,g) \\ \quad+\el(g^{-1}x,t)\end{array}\right)
  \,d\mu^{*m}(s)\,d\mu^{*n}(t)\,d(x,g)\\
&\le 4C^2D.
\end{align*}
We note that $f_R$ is monotone increasing in $R$ and
define
\[
h(x)=\frac{1}{2}C^2\lim_{R\to\infty}f_R(x)+C^2\liminf_{n\to\infty}\int_G \el(x,s)+\el(s^{-1}x,s^{-1})\,d\mu^{*n}(s).
\]
By Fatou \& Fubini, $h\in L^1(X)$ with $\|h\|_1\le 4C^4D$.
Moreover, one has
\begin{align*}
\el_R & (x,g) = \liminf_{m,n}\int_{G^2} \el_R(x,g)\,d\mu^{*m}(s)\,d\mu^{*n}(t)\\
&\le C^2\liminf_{m,n}\int_{G^2}\left(\begin{array}{l}
 \el_R(x,s^{-1})+\el_R(sx,sgt)\\ \quad+\el_R(t^{-1}g^{-1}x,t^{-1})\end{array}\right)
 \,d\mu^{*m}(s)\,d\mu^{*n}(t)\\
&\le C^2\liminf_{m,n}\left(\begin{array}{l}\int_{G} \el(x,s)\,d\mu^{*m}(s)+(\mu^{*m}*\el_R*\mu^{*n})(x,g)\\
    \quad+\int_{G} \el(t^{-1}g^{-1}x,t^{-1})\,d\mu^{*n}(t)\end{array}\right)\\
&\le h(x)+h(g^{-1}x)
\end{align*}
for a.e.\ $(x,g)\in X\times G$ and $R>0$.
This completes the proof.
\end{proof}

Now, let $\G\le G$ be a lattice.
By rescaling, we assume that $X=G/\G$ is a probability $G$-space.
Choose a Borel lifting $\sigma\colon X\to G$ and denote by $\beta\colon X\times G\to\G$
the associated cocycle given by
\[
\beta(x,g)=\sigma(x)^{-1}g\sigma(g^{-1}x).
\]
It satisfies the cocycle relation
\[
\beta(x,gh)=\beta(x,g)\beta(g^{-1}x,h)\mbox{ and }\beta(x,g)^{-1}=\beta(g^{-1}x,g^{-1}).
\]
We note the following fact, which has its own interest and can be used
to prove that property {\TTT} is inherited to cocompact lattices.
\begin{cor}
Let $\G\le G$ be a lattice and $\el\colon\G\to\IR_{\geq0}$ be a function
such that $\el(gh)\le\el(g)+\el(h)$ for all $g,h\in \G$, and let
\[
L(g):=\int_{G/\G}\el(\beta(x,g))\,dx\in [0,+\infty].
\]
If $L$ is essentially bounded, then $\el$ is bounded.
\end{cor}
\begin{proof}
We consider the function $\el(\beta(x,g))$ on the groupoid $X\times G$,
where $X=G/\G$.
By Theorem~\ref{thm:oid}, there is $h\in L^1(X)$ such that
$\|h\|_1\le 4\|L\|_\infty$ and $\el(\beta(x,g))\le h(x)+h(g^{-1}x)$.
Let $X_0=\{ x : h(x)\le5\|L\|_\infty\}$, which is non-negligible.
Then, for every $s\in\G$ and a.e.\ $x,y\in X_0$, one has
\[
\el(s)=\el(\beta(x,\sigma(x)s\sigma(y)^{-1}))\le h(x)+h(y)\le 10\|L\|_\infty.
\]
This completes the proof.
\end{proof}

\begin{proof}[Proof of Theorem~\ref{thm:lattice}]
First, we suppose that $G$ has property {\Tp} and prove $\G$ has the same.
Let $\e>0$ be given and take $(K,\delta)$ which satisfies condition {\Tp} for $G$.
We may assume that the lifting $\sigma\colon X\to G$ is regular in the sense that
it maps a compact subset of $X=G/\G$ to a relatively compact subset of $G$.
Choose a compact subset $X_0\subset X$ whose measure is at least $1-\delta/4$,
and let $F=\{\beta(x,g) : x\in X_0,\ g\in K\}$, which is a finite subset in $\G$.
We will prove that $(F,\delta/2)$ satisfies condition {\Tp} for $\G$.
To do so, let $\theta\colon\G\times\G\to\IC$ be a normalized positive
definite kernel such that
\[
\sup_{s\in\G}\| s\cdot\theta - \theta \|_{\cb} < \delta/2
\ \mbox{ and }\
\sup_{s^{-1}t\in F} | \theta(s,t) - 1 | < \delta/2.
\]
We induce $\theta$ from $\G$ to $G$ by defining
\[
\tilde{\theta}(g,h)=\int_X\theta(\beta(x,g),\beta(x,h))\,dx.
\]
Then, $\tilde{\theta}$ is a normalized Borel positive definite kernel such that
\[
\sup_{g\in G}\|g\cdot\tilde{\theta}-\tilde{\theta}\|_{\cb}
\le\sup_{g\in G}\int_X\|\beta(x,g)\cdot\theta-\theta\|\,dx < \delta/2
\]
and, since $\beta(x,g)^{-1}\beta(x,h)=\beta(g^{-1}x,g^{-1}h) \in F$ for
$x\in gX_0$ and $g^{-1}h\in K$,
\[
\sup_{g^{-1}h\in K}|\tilde{\theta}(g,h)-1|
\le \sup_{g^{-1}h\in K}\int_{gX_0}|\theta(\beta(x,g),\beta(x,h))-1|\,dx +\delta/2< \delta.
\]
It follows from property {\Tp} that
\[
\sup_{g\in G}|\tilde{\theta}(g,1)-1|<\e.
\]
We express $\theta$ as $\theta(s,t)=\ip{P(s),P(t)}$ and define $\el\colon X\times G\to\IR_{\geq0}$
by
\[
\el(x,g)=\|P(\beta(x,g))-P(1)\| + \delta^{\oh}.
\]
Since $\|P(st)- P(s)\|^2 < \|P(t)-P(1)\|^2+\delta$ for all $s,t\in\G$, one has
\begin{align*}
\el(x,gh) &\le \el(x,g) + \|P(\beta(x,g)\beta(g^{-1}x,h))-P(\beta(x,g))\|\\
 &\le \el(x,g) + \el(g^{-1}x,h).
\end{align*}
Moreover,
\[
\int_X\el(x,g)\,dx\le\left(\int_X\|P(\beta(x,g))-P(1)\|^2\,dx\right)^{\oh}
  + \delta^{\oh} < 3\e^{\oh}
\]
for all $g\in G$.
By Theorem~\ref{thm:oid}, there is $h\in L^1(X)$ such that
$\|h\|_1\le 12\e^{\oh}$ and $\el(x,g)\le h(x)+h(g^{-1}x)$ a.e.
Let $X_0=\{ x : h(x)<13\e^{\oh}\}$, which is non-negligible.
Then, for every $s\in\G$ and a.e.\ $x,y\in X_0$, one has
\[
|1-\theta(s,1)|\le\frac{1}{2}\|P(s)-P(1)\|^2
 \le\frac{1}{2}\el(x,\sigma(x)s\sigma(y)^{-1})^2 <100\e.
\]
This proves that $\G$ has property {\Tp}.
We just mention that the proof of measure-equivalence invariance of property {\Tp}
is similar to above.

Next, we suppose $\G$ has property {\Tp} and prove $G$ has the same.
Let $\e>0$ be given and take $(F,\delta)$ which satisfies condition {\Tp} for $\G$.
We take a compact subset $K\subset G$ such that $F\subset K$ and
$|K\cap\sigma(X)|>1-\e$.
We will prove that $(K,\delta)$ satisfies condition {\Tp} for $G$.
To do so, let $\theta\colon G\times G\to\IC$ be a normalized continuous
positive definite kernel such that
\[
\sup_{g\in G}\|g\cdot\theta-\theta\|_{\cb}<\delta
\ \mbox{ and }\
\sup_{g^{-1}h\in K}|\theta(g,h)-1| < \delta.
\]
Then, property {\Tp} implies
\[
\sup_{s\in\G}|\theta(s,1)-1| < \e.
\]
It follows (see Section~\ref{sec:prelim}) that for any $g\in G$, $y\in X$ and $s\in\G$, one has
\begin{align*}
|\theta(g\sigma(y)s,g)-1| %&\le \delta+|\theta(\sigma(y)s,1)-1|\\
&< \delta+|\theta(\sigma(y)s,\sigma(y))-1|+\sqrt{2}|\theta(\sigma(y),1)-1|^{1/2}\\
&< 2\delta+\e+\sqrt{2}|\theta(\sigma(y),1)-1|^{1/2}.
\end{align*}
Hence,
\[
|1 - \int_X \theta(g\sigma(x)\beta(x,g^{-1}),g)\,dx | < 6\e^{1/2}
\]
for every $g\in G$. On the other hand, since $g\sigma(x)\beta(x,g^{-1})=\sigma(gx)$,
\[
\int_X \theta(g\sigma(x)\beta(x,g^{-1}),g)\,dx
= \int_X \theta(\sigma(gx),g)\,dx\approx_{\e+(2\delta)^{1/2}} \theta(1,g).
\]
Therefore,
$|\theta(g,1)-1|\le 10\e^{\oh}$ for all $g\in G$.
This completes the proof.
%Express $\theta$ as $\theta(g,h)=\ip{P(g),P(h)}$.
%Then,
%\begin{align*}
%\|P(g\sigma(y)s)-P(g)\|^2 &\le \| P(\sigma(y)s) - P(1) \|^2 + 2\delta\\
%&\le \|P(s)-P(1)\|^2+\|P(\sigma(y))-P(1)\|^2+4\delta
%\end{align*}
%for all $g\in G$, $y\in X$ and $s\in G$.
%It follows that
%\[
%\|P(g) - \int_X P(g\sigma(x)\beta(x,g^{-1}))\,dx\|\le 10\e^{\oh}
%\]
%for every $g\in G$. On the other hand,
%\[
%\int_X P(g\sigma(x)\beta(x,g^{-1}))\,dx=\int_X P(\sigma(gx))\,dx=\int_X P(\sigma(x))\,dx\approx_{2\e^{\oh}} P(1).
%\]
%Therefore,
%$\|P(g)-P(1)\|\le 14\e^{\oh}$ and hence
%$|\Re\theta(g,h)-1|<500\e$ for all $g,h\in G$.
%This completes the proof.
\end{proof}

\begin{proof}[Proof of Corollary~\ref{cor:ug}]
Let $\e>0$ be given and take $(K,\delta)$ which satisfies condition {\Tq}.
Let $\pi'\colon G'\to\cU(\cH)$ be a $C_0$ unitary representation
which has approximately $G'$-invariant unit vectors $\xi_n$.
(See Theorem 2.1.1 in \cite{ccjjv}).
We consider $\pi\colon G\to\cU(L^2(X,\cH))$ defined by
\[
(\pi(g)\xi)(x)=\pi'(\beta(x,g))\xi(g^{-1}x),
\]
(see the remark at the end of this proof) and let
\[
D_n(x)=\sup_{g,h\in G}\|\xi_n-\pi'\big(\beta(x,gh)^{-1}\beta(x,g)\beta(g^{-1}x,h)\big)\xi_n\|.
\]
We view $\xi_n\in\cH$ as constant vectors in $L^2(X,\cH)$.
Since $D_n(x)\le 2$ and $D_n(x)\to0$ for a.e.\ $x\in X$ by  assumption, one has
\[
\sup_{g,h\in G}\| \pi(gh)\xi_n-\pi(g)\pi(h)\xi_n\|^2
\le \int_X D_n(x)^2\,dx \to 0,
\]
and
\[
\sup_{g\in K}\|\pi(g)\xi_n-\xi_n\|^2=\sup_{g\in K}\int_X \|\pi'(\beta(x,g))\xi_n-\xi_n\|^2\,dx\to0.
\]
Hence by property {\Tq}, there is $n$ such that $\xi=\xi_n$ satisfies
\[
\sup_{g\in G}\|\pi(g)\xi-\xi\|<\e\mbox{ and }\int_X D(x)\,dx<\e.
\]
Then,
$\el(x,g)=\|\pi'(\beta(x,g))\xi-\xi\|+D(x)$
satisfies $\el(x,gh)\le 2\el(x,g)+\el(g^{-1}x,h)$ and $\sup_g\int_X\ell(x,g)\,dx\le 2\e$.
Hence, by Theorem~\ref{thm:oid}, there is $h\in L^1(X)$ such that
$\|h\|_1\le 2^7\e$ and $\el(x,g)\le h(x)+h(g^{-1}x)$ a.e.
Then, $X'=\{ x : h(x)<1/4\}$ has measure at least $1-2^9\e$.
Since $\pi'$ is a $C_0$-representation, $\{\beta(x,g) : x,g^{-1}x\in X'\}$
is relatively compact in $G'$.

\emph{Remark.} The map $\pi$, defined as above, is in general Haar measurable
instead of Borel measurable. To fix this problem, either go through all proofs
in this paper with measurable maps and $\esup$ in place of Borel maps
and $\sup$, or take an ad hoc measure as follows: there is a null set $N$ such
that $\pi$ is Borel on $G\setminus N$. Let $K$ be any compact neighborhood of $G$.
By the Lusin--Novikov uniformization theorem, one can find a Borel map
$t\colon G\to K$ such that $gt_g^{-1},t_g\in G\setminus N$ for all $g\in G$.
Now, replace $\pi(g)$ with $\pi(gt_g^{-1})\pi(t_g)$, which is a Borel map.
\end{proof}

\end{document}